\documentclass[10pt, leqno]{amsart}
\usepackage{amsmath} 
\usepackage{amsthm}
\usepackage{xypic}
\usepackage{pifont}

\theoremstyle{plain}
\newtheorem{Lem}{Lemma}[section]

\newtheorem{Cor}[Lem]{Corollary}
\newtheorem{Thm}[Lem]{Theorem}

{\theoremstyle{definition} 

\newtheorem{Rk}[Lem]{Remark}
\newtheorem{Def}[Lem]{Definition}}

\newcommand{\zig}{\addtocounter{Lem}{1}\tag{\theLem}}
 
\def\:{\colon} 
\DeclareMathOperator*{\colim}{colim}
\DeclareMathOperator*{\holim}{holim}
\DeclareMathOperator*{\holimG}{holim^\mathit{G}}
\DeclareMathOperator*{\holimH}{holim^\mathit{H}}

\begin{document}
\title{Explicit fibrant replacement for discrete $G$-spectra}

\author{Daniel G. Davis}
\begin{abstract}
If $\mathcal{C}$ is the model category of simplicial presheaves on a 
site with enough points, with fibrations equal to the global fibrations, 
then it is well-known that the fibrant objects are, in general, mysterious. Thus, it is not 
surprising that, when $G$ is a profinite group, the fibrant objects in the model 
category of discrete $G$-spectra are also difficult to get a handle on. However, 
with simplicial presheaves, it is possible to construct an explicit fibrant model 
for an object in $\mathcal{C}$, under certain finiteness 
conditions. Similarly, in this paper, we show that if $G$ has 
finite virtual cohomological dimension and $X$ is a discrete $G$-spectrum, then 
there is an 
explicit fibrant model for $X$. Also, we give several applications of this 
concrete model related to closed subgroups of $G$.  
\end{abstract}

\maketitle

\section{Introduction}
 
\par
In this paper, $G$ always denotes a profinite group and, by ``spectrum," we 
mean a Bousfield-Friedlander spectrum of simplicial sets. 
In particular, a {\em discrete $G$-spectrum} is a $G$-spectrum such that each 
simplicial set $X_k$ is a simplicial object in the category of discrete 
$G$-sets (thus, the action map on the 
$l$-simplices, \[G \times (X_k)_l 
\rightarrow (X_k)_l,\] is continuous when $(X_k)_l$ is regarded as a discrete 
space, for all $l \geq 0$). The category of discrete $G$-spectra, with morphisms being 
the $G$-equivariant maps of spectra, is denoted by $\mathrm{Spt}_G$. 
\par
As shown in \cite[Section 3]{cts}, $\mathrm{Spt}_G$ is a simplicial 
model category, where a morphism 
$f$ in $\mathrm{Spt}_G$ is a weak equivalence (cofibration) if and only if 
$f$ is a weak equivalence (cofibration) in $\mathrm{Spt}$, the simplicial model 
category of 
spectra. Given $X \in \mathrm{Spt}_G$, 
the {\em homotopy fixed point 
spectrum} $X^{hG}$ is the total right derived functor of fixed 
points: $X^{hG} = (X_{f,G})^G,$ where $X \rightarrow X_{f,G}$ is a trivial 
cofibration and $X_{f,G}$ is fibrant, all in 
$\mathrm{Spt}_G$. This definition generalizes the classical definition of 
homotopy fixed point spectrum, in the case when $G$ is a finite group (see 
\cite[pg. 337]{cts}). 

\par
Notice that we can loosen up the requirements on $X_{f,G}$. If 
$X \rightarrow X^f$ is a weak equivalence, with $X^f$ fibrant, all in $\mathrm{Spt}_G$, then, by the right lifting property of a fibrant object, there is a 
weak equivalence $X_{f,G} \rightarrow X^f$ in $\mathrm{Spt}_G$, so that 
\[X^{hG} = (X_{f,G})^G \rightarrow (X^f)^G\] is a weak equivalence. Thus, 
we can identify $X^{hG}$ and $(X^f)^G$, and, hence, we only need a fibrant 
replacement $X^f$ to form 
$X^{hG}$. Henceforth, we relabel any such $X^f$ as $X_{f,G}$ and refer to it as a 
{\it globally fibrant model} for $X$. (Thus, from now on, $X \rightarrow X_{f,G}$ does 
not have to be a cofibration.)

\par
The preceding discussion shows that a globally fibrant model $X_{f,G}$ is 
an important object. Of course, the model category axioms guarantee that 
$X_{f,G}$ always exists. But it is reasonable to ask for more. For example, 
in $\mathrm{Spt}$, there is a 
functor \[Q \: \mathrm{Spt} \rightarrow \mathrm{Spt}, \ \ \ Z \mapsto Q(Z) = 
Z_\mathtt{f},\] where $Z_\mathtt{f}$ is a fibrant spectrum, 
\[(Z_\mathtt{f})_k = \colim_n \Omega^n (\mathrm{Ex}^\infty(Z_{k+n})),\] and 
there is a natural weak equivalence $Z \rightarrow Z_\mathtt{f}$ (see, 
for example, \cite[pg. 524]{Thomason}). Hence, for the model category 
$\mathrm{Spt}$, there is always an explicit model for fibrant replacement. Similarly, 
it is natural to wonder if an explicit model 
for $X_{f,G}$ is available.  
\par
But there is a difficulty with this. Let 
$G\negthinspace-\negthinspace\mathbf{Sets}_{df}$ be the Grothendieck site 
of finite discrete $G$-sets (e.g., see \cite[Section 6.2]{Jardine}). 
There is an equivalence between 
$\mathrm{Spt}_G$ and the category of 
sheaves of spectra on 
$G\negthinspace-\negthinspace\mathbf{Sets}_{df}$ (the discrete $G$-spectrum 
$X$ corresponds to the sheaf of spectra $\mathrm{Hom}_G(-,X)$; 
see \cite[Section 3]{cts} for details), 
and it is well-known that, in general, for categories of simplicial presheaves, 
presheaves of spectra, and sheaves of spectra, there is no known explicit 
model for a globally fibrant object. In fact, the situation is such that 
\cite[pg. 1049]{GJlocal} says 
that ``[t]he fibrant objects in all of these theories continue to be really quite 
mysterious" (a similar statement 
appears in \cite[between Corollary 19 and Definition 20]{Jardinesummary}).

\par
Nevertheless, under certain hypotheses, explicit models 
for globally fibrant objects are available in the cases of 
simplicial presheaves and presheaves of spectra. Such 
results are based on Jardine's result in 
\cite[Proposition 3.3]{Jardinejpaa}, which constructs an 
explicit globally fibrant model for a simplicial presheaf $P$ on 
the site \'et$\,|_S$, where $P$ and the scheme $S$ must satisfy certain finiteness conditions (and other hypotheses). For example, under similar 
finiteness conditions, \cite[Proposition 3.20]{Mitchell} 
follows the proof of Jardine's result to obtain 
a concrete globally fibrant model for a presheaf of 
spectra on a site with enough points. 

\par
Now suppose that $G$ has finite virtual cohomological dimension (see Definition \ref{vcd}) and that $X$ is a discrete $G$-spectrum. In this paper, we show that there is an explicit model for 
$X_{f,G}$, by expressing the homotopy limit for diagrams in 
$\mathrm{Spt}_G$ in terms of the 
homotopy limit for diagrams in $\mathrm{Spt}$ (see Theorem \ref{iso}) and 
by modifying the 
proof of \cite[Theorem 7.4]{cts} (which 
applies the two results cited above, \cite[Proposition 3.3]{Jardinejpaa} and 
\cite[Proposition 3.20]{Mitchell}). We refer the reader to Theorem \ref{main} 
for the precise statement of our main result; its formulation depends on definitions that are given in Section \ref{model}.

\par
Let $H$ be a closed subgroup of $G$. If $Y \rightarrow Z$ is a weak equivalence in 
$\mathrm{Spt}_H$, such that $Y$ is a globally 
fibrant model for $X$ in $\mathrm{Spt}_G$ 
and $Z$ is a globally fibrant model for $X$ in $\mathrm{Spt}_H$, then we label the map $Y \rightarrow Z$ as $r^G_H$. 
Note that the map \[X_{f,G} \rightarrow (X_{f,G})_{f,H},\] a 
weak equivalence in $\mathrm{Spt}_H$, can be labelled as $r^G_H$, so that 
$r^G_H$ always exists. 
In Corollary \ref{mapGH}, 
we show that the explicit globally fibrant model constructed in Theorem \ref{main} 
yields an 
explicit model for $r^G_H$ (where, as before, we assume that 
$G$ has finite virtual cohomological dimension).
\par
Section \ref{hH} explains that, when $H$ is a closed normal subgroup of 
$G$ and $X$ is a discrete $G$-spectrum, 
there are cases when $X^{hH}$, unlike $X^H$, is not known to be a 
$G/H$-spectrum. In Corollary \ref{G/H}, we point out 
that, if $G$ has finite virtual cohomological 
dimension, then Theorem 
\ref{main} implies that $X^{hH}$ can always be taken to be a $G/H$-spectrum.
\par
Throughout this paper, $U <_o G$ means that $U$ is an open subgroup of 
$G$.
\vspace{.1in}
\par
\noindent
\textbf{Acknowledgements.} I thank Mark Hovey for a helpful conversation. 
 
\section{Homotopy limits in the category of discrete $G$-spectra} 

\par 
To explicitly construct the desired fibrant discrete $G$-spectrum, we first need to 
understand homotopy limits in the category of discrete $G$-spectra. 
Thus, we begin this section by following the presentation in 
\cite{Hirschhorn} to give the general 
definition of the homotopy limit of a {\em $\mathcal{C}$-diagram} $X(-)$ in $\mathcal{M}$, where $\mathcal{M}$ is a simplicial model category and $X(-)$ is a diagram in 
$\mathcal{M}$ indexed by a small category $\mathcal{C}$. 

\par
Recall that, for a small category $\mathcal{C}$, the {\em classifying space} of $\mathcal{C}$ 
is the simplicial set $B\mathcal{C}$, with $l$-simplices $(B\mathcal{C})_l$ equal to 
the set of compositions 
\[c_0 \overset{\sigma_0}{\longrightarrow} c_1 
\overset{\sigma_1}{\longrightarrow} 
\cdots \overset{\sigma_{l-1}}{\longrightarrow} c_l\] in $\mathcal{C}$ (see, for example, 
\cite[Definition 14.1.1]{Hirschhorn} for the definition of the face and degeneracy maps). 

\begin{Def}[{\cite[Definition 18.1.8]{Hirschhorn}}]\label{holim}
As above, let $\mathcal{M}$ be a simplicial model category and let $\mathcal{C}$ be a small category. Also, let $X = X(-)$ be a $\mathcal{C}$-diagram in 
$\mathcal{M}$; that is, $X$ is a functor $\mathcal{C} \rightarrow \mathcal{M}$, so 
that, for example, if $c \rightarrow d$ is a morphism in $\mathcal{C}$, then $X(c) \rightarrow X(d)$ is a morphism in $\mathcal{M}$. 
Then the {\em homotopy limit} of $X$ in $\mathcal{M}$,
$\holim_\mathcal{C}^\mathcal{M} X$, is defined to be the equalizer of the 
diagram 
\[\xymatrix{ 
\prod_{c \in \mathcal{C}} X(c)^{B(\mathcal{C} \downarrow c)} 
\ar@<1ex>[r]^-{\alpha} \ar@<-.5ex>[r]_-{\beta} &
\prod_{\sigma \: c \rightarrow d} X(d)^{B(\mathcal{C} \downarrow c)}
},\] 
where the second product is indexed over all the morphisms in $\mathcal{C}$. Here, the 
map $\alpha$ is defined as follows: the projection of $\alpha$ onto the factor indexed by 
$\sigma \: c_0 \rightarrow c_1$ is equal to composing the projection 
\[\textstyle\prod_{c \in \mathcal{C}} X(c)^{B(\mathcal{C} \downarrow c)} \rightarrow 
X(c_0)^{B(\mathcal{C} \downarrow c_0)}\] with the canonical map 
$X(c_0)^{B(\mathcal{C} \downarrow c_0)} 
\rightarrow X(c_1)^{B(\mathcal{C} \downarrow c_0)}.$ The 
map $\beta$ is defined by letting the projection of $\beta$ onto the factor indexed by 
$\sigma$ be given by composing the projection 
\[\textstyle\prod_{c \in \mathcal{C}} X(c)^{B(\mathcal{C} \downarrow c)} \rightarrow 
X(c_1)^{B(\mathcal{C} \downarrow c_1)}\] with the canonical map 
$X(c_1)^{B(\mathcal{C} \downarrow c_1)} \rightarrow X(c_1)^{B(\mathcal{C} 
\downarrow c_0)}$ 
that is induced by the map 
$B(\mathcal{C} \negthinspace \downarrow \negthinspace c_0) 
\rightarrow B(\mathcal{C} \negthinspace \downarrow \negthinspace c_1).$  
\end{Def} 
 
\begin{Rk}\label{commute}
Note that the homotopy limit is an equalizer of a diagram involving products and 
cotensors, and, given a simplicial set $K$, the cotensor functor $(-)^K \: \mathcal{M} 
\rightarrow \mathcal{M}$ is a right adjoint. Hence, the homotopy limit 
$\holim_\mathcal{C}^\mathcal{M} (-)$ commutes with limits in 
$\mathcal{M}$.
\end{Rk}
 
\par
If $X$ and $Z$ are $\mathcal{C}$-diagrams in $\mathrm{Spt}_G$ and 
$\mathrm{Spt}$, respectively, then we use the less 
cumbersome $\holim_\mathcal{C}^G X$ and 
$\holim_\mathcal{C} Z$ to denote $\holim_\mathcal{C}^{\mathrm{Spt}_G} X$ and 
$\holim_\mathcal{C}^\mathrm{Spt} Z$, respectively. As 
in Definition \ref{holim}, given $c \in \mathcal{C}$, $X(c)$ is the object in 
$\mathcal{M}$ that is indexed by $c$ in the diagram $X$. 

\begin{Thm}\label{iso}
If $X$ is a $\mathcal{C}$-diagram in $\mathrm{Spt}_G$, then there is an 
isomorphism 
\[\holimG_\mathcal{C} X \cong 
\colim_{N \vartriangleleft_o G} (\holim_\mathcal{C} X)^N.\]
\end{Thm}

\begin{proof}
For each $c \in \mathcal{C}$, let $B_c$ denote the simplicial set 
$B(\mathcal{C} \negthinspace \downarrow \negthinspace c)$. Given a discrete 
$G$-spectrum $Y$ and a simplicial set $K$, let $(Y^K)_G$ and $Y^K$ 
denote the cotensor objects in $\mathrm{Spt}_G$ and $\mathrm{Spt}$, 
respectively. Also, let $\prod^G$ and $\prod$ denote products in $\mathrm{Spt}_G$ and 
$\mathrm{Spt}$, respectively. 
Then, by Definition \ref{holim}, $\holim_\mathcal{C}^G X$ is the equalizer of the 
diagram 
\[\xymatrix{ 
\prod^G_{c \in \mathcal{C}} (X(c)^{B_c})_G
\ar@<1ex>[r]^-{\alpha} \ar@<-.5ex>[r]_-{\beta} &
\prod^G_{\sigma \: c \rightarrow d} (X(d)^{B_c})_G
}\] in $\mathrm{Spt}_G$.  

\par
To go further, we note how limits and cotensors in 
$\mathrm{Spt}_G$ are formed. Recall from \cite[Remark 4.2]{cts} that, if 
$\{Y_\alpha\}_\alpha$ is any diagram in $\mathrm{Spt}_G$, 
then the limit of this 
diagram in the category $\mathrm{Spt}_G$ is 
$\colim_{N \vartriangleleft_o G} (\lim_\alpha Y_\alpha)^N,$ where the limit in 
this expression is taken in the category of spectra. The colimit in this expression, and 
others like it, can be taken in the category of spectra, since the forgetful functor 
$\mathrm{Spt}_G \rightarrow \mathrm{Spt}$ is a left adjoint, by \cite[Corollary 3.8]{cts}. 
Given 
a discrete $G$-spectrum $Y$ and a simplicial set $K$, the spectrum $Y^K$ can be 
regarded as a $G$-spectrum (but not necessarily a discrete $G$-spectrum), 
by using only the $G$-action on $Y$. Then 
\[(Y^K)_G = \colim_{N \vartriangleleft_o G} (Y^K)^N\] (e.g., see 
\cite[(1.2.2)]{hGal} and \cite[pg. 42]{Jardinejpaa}). We apply these observations 
as follows. 
\par
First of all, note that $\holim_\mathcal{C}^G X$ is the equalizer in $\mathrm{Spt}_G$ of 
the diagram 
\begin{displaymath}
\xymatrix{ 
\displaystyle{\colim_{N \vartriangleleft_o G}}(\textstyle\prod_{c \in \mathcal{C}} 
(X(c)^{B_c})_G)^N
\ar@<1ex>[r]^-{\alpha} \ar@<-.5ex>[r]_-{\beta} & 
\displaystyle{\colim_{N \vartriangleleft_o G}}(\textstyle\prod_{\sigma \: c \rightarrow d} 
(X(d)^{B_c})_G)^N}.
\end{displaymath}
Furthermore, 
let $S$ be a $G$-set (but not necessarily a discrete $G$-set) and let $U$ be an open normal subgroup of $G$. Then it 
is clear that $(\bigcup_{N \vartriangleleft_o G} S^{N})^U \subset S^U$. 
Since $U \in \{N \, | \, N \vartriangleleft_o G\}$, $S^U \subset 
\bigcup_{N \vartriangleleft_o G} S^{N},$ and, hence, $S^U \subset 
(\bigcup_{N \vartriangleleft_o G} S^{N})^U.$ Thus, we can conclude that 
\[S^U = (\textstyle\bigcup_{N \vartriangleleft_o G} S^{N})^U.\] Similarly, if $Y$ is a $G$-spectrum, 
\[Y^U \cong  (\colim_{N \vartriangleleft_o G} Y^{N})^U.\] Therefore, 
\[(\textstyle\prod_{c \in \mathcal{C}} (X(c)^{B_c})_G)^U = 
\textstyle\prod_{c \in \mathcal{C}} 
(\displaystyle\colim_{N \vartriangleleft_o G} (X(c)^{B_c})^N)^U \cong 
\textstyle\prod_{c \in \mathcal{C}} (X(c)^{B_c})^{U},\] and, similarly, 
\[(\textstyle\prod_{\sigma \: c \rightarrow d} (X(d)^{B_c})_G)^U \cong 
\textstyle\prod_{\sigma \: c \rightarrow d} (X(d)^{B_c})^{U}.\]

\par
The preceding 
two isomorphisms imply that $\holim_\mathcal{C}^G X$ is isomorphic to 
the equalizer in 
$\mathrm{Spt}_G$ of the diagram 
\[\xymatrix{ 
\displaystyle\colim_{N \vartriangleleft_o G}(\textstyle\prod_{c \in \mathcal{C}} 
X(c)^{B_c})^N 
\ar@<1ex>[r]^-{\alpha} \ar@<-.5ex>[r]_-{\beta} &
\displaystyle\colim_{N \vartriangleleft_o G} 
(\textstyle\prod_{\sigma \: c \rightarrow d} X(d)^{B_c})^N.}\] 
Thus, $\holim_\mathcal{C}^G X \cong \colim_{N \vartriangleleft_o G} \mathcal{E}^N$, 
where $\mathcal{E}$ is the equalizer in $\mathrm{Spt}$ of the diagram
\[\xymatrix{ 
\textstyle{\colim_{N \vartriangleleft_o G}} (\textstyle\prod_{c \in \mathcal{C}} 
X(c)^{B_c} 
\ar@<.5ex>[r]^-{\alpha'} \ar@<-.6ex>[r]_-{\beta'} &
\textstyle\prod_{\sigma \: c \rightarrow d} X(d)^{B_c})^N,}\] 
where $\alpha'$ and $\beta'$ are the maps in the equalizer diagram for 
$\holim_\mathcal{C} X$. 

\par
Since filtered colimits and finite limits commute, 
$\mathcal{E} \cong \colim_{N \vartriangleleft_o G} (\mathcal{E}')^N,$ where 
$\mathcal{E}'$ is the equalizer in $\mathrm{Spt}$ of the diagram 
\[\xymatrix{\prod_{c \in \mathcal{C}} 
X(c)^{B_c} \ar@<.5ex>[r]^-{\alpha'} \ar@<-.6ex>[r]_-{\beta'} & 
\prod_{\sigma \: c \rightarrow d} X(d)^{B_c}.}\] Notice that 
$\mathcal{E}' = \holim_\mathcal{C} X$. 
\par
If $U$ is an open normal subgroup of $G$, then 
$(\holim_\mathcal{C} X)^U$ is a $G/U$-spectrum, so that 
$\colim_{N \vartriangleleft_o G} (\holim_\mathcal{C} X)^N$ is a discrete 
$G$-spectrum. Also, given $Y \in \mathrm{Spt}_G$, 
there is an isomorphism $\colim_{N \vartriangleleft_o G} Y^N \cong Y$. Therefore, 
putting our various observations together, we obtain that
\begin{align*}
\holimG_\mathcal{C} X & \cong 
\colim_{N \vartriangleleft_o G} \mathcal{E}^N 
\cong \colim_{N \vartriangleleft_o G}(\colim_{N' 
\vartriangleleft_o G}(\mathcal{E}')^{N'})^N 
\\ & \cong \colim_{N' \vartriangleleft_o G}(\mathcal{E}')^{N'} 
= \colim_{N \vartriangleleft_o G}(\holim_\mathcal{C} X)^N.
\end{align*}
\end{proof}

\par
If $X$ is a $\mathcal{C}$-diagram of fibrant discrete $G$-spectra (that is, 
$X(c)$ is fibrant in $\mathrm{Spt}_G$, for all $c \in \mathcal{C}$), then 
$\holim_\mathcal{C}^G X$ is a fibrant discrete $G$-spectrum, by 
\cite[Theorem 18.5.2 (2)]{Hirschhorn}, so that Theorem \ref{iso} gives the 
following result.

\begin{Cor}\label{maincor}
If $X$ is a $\mathcal{C}$-diagram of fibrant discrete $G$-spectra, then the 
spectrum 
$\colim_{N \vartriangleleft_o G} (\holim_\mathcal{C} X)^N$ is a fibrant 
discrete $G$-spectrum.
\end{Cor}

\par 
We conclude this section with some observations about Corollary \ref{maincor}.

\begin{Def}\label{presheaf}
If $P$ is a $\mathcal{C}$-diagram of 
presheaves of spectra on the site 
$G\negthinspace-\negthinspace\mathbf{Sets}_{df}$, then there is a presheaf of spectra 
$\holim_\mathcal{C} P,$ defined by 
\[(\holim_\mathcal{C} P)(S) = \holim_\mathcal{C} P(S),\] for each 
$S \in G\negthinspace-\negthinspace\mathbf{Sets}_{df}$.
\end{Def}

\par
Let $X$ be a $\mathcal{C}$-diagram in $\mathrm{Spt}_G$. Then it is natural 
to form the presheaf of spectra $\holim_\mathcal{C} \mathrm{Hom}_G(-,X)$. 
Also, let 
\[\mathcal{F} = 
\mathrm{Hom}_G(-, \colim_{N \vartriangleleft_o G} 
(\holim_\mathcal{C} X)^N),\] 
the canonical sheaf of spectra on the site 
$G\negthinspace-\negthinspace\mathbf{Sets}_{df}$ associated to the spectrum 
$\colim_{N \vartriangleleft_o G} (\holim_\mathcal{C} X)^N$ that is considered in 
Corollary \ref{maincor}.  
\par
The following lemma says that the presheaf 
$\holim_\mathcal{C} \mathrm{Hom}_G(-,X)$ is actually a sheaf of spectra, since it 
is isomorphic to $\mathcal{F}$. 

\begin{Lem}\label{isopresheaves}
If $X$ is a $\mathcal{C}$-diagram of discrete $G$-spectra, then the 
presheaves of spectra $\mathcal{F}$ and $\holim_\mathcal{C} \mathrm{Hom}_G(-,X)$ 
on the site $G\negthinspace-\negthinspace\mathbf{Sets}_{df}$ are 
isomorphic.
\end{Lem}

\begin{proof}
Let $S$ be a finite discrete $G$-set: $S$ can be identified with a 
disjoint union $\coprod_{j=1}^m G/U_j$, where each $U_j$ is an open subgroup of 
$G$. Notice that the collection $\{N\}_{N \vartriangleleft_o G}$ of open normal 
subgroups of $G$ is a cofinal subcollection 
of the collection $\{U\}_{U <_o G}$ of open subgroups of $G$, so that, if $Y$ 
is a $G$-spectrum, there is an isomorphism 
$\textstyle\colim_{N \vartriangleleft_o G} Y^N \cong 
\textstyle\colim_{U <_o G} Y^U$ of $G$-spectra. Hence, 
\begin{align*}
\mathcal{F}(S)
& \cong \textstyle\prod_{j=1}^m 
(\displaystyle\colim_{N \vartriangleleft_o G} (\holim_\mathcal{C} 
X)^N)^{U_j} \\ & \cong \textstyle\prod_{j=1}^m 
(\displaystyle\colim_{U <_o G} (\holim_\mathcal{C} X)^U)^{U_j} \\ &
\cong \textstyle\prod_{j=1}^m (\displaystyle\holim_\mathcal{C} X)^{U_j} \\ & \cong 
\holim_\mathcal{C} \mathrm{Hom}_G(S, X),
\end{align*} 
where the third isomorphism is due to the fact that $U_j \in \{U \, | \, U<_o G\}$ 
(as in the proof of Theorem \ref{iso}) and the last isomorphism applies 
Remark \ref{commute}. This chain of isomorphisms shows that there is an 
isomorphism 
$\mathcal{F}(S) \cong \textstyle\holim_\mathcal{C} \mathrm{Hom}_G(S, X)$ that 
is natural for $S \in G\negthinspace-\negthinspace\mathbf{Sets}_{df}$, 
so that $\mathcal{F}$ and $\holim_\mathcal{C} \mathrm{Hom}_G(-,X)$ are 
isomorphic presheaves of spectra.
\end{proof}
 
\begin{Rk}
Let $X$ be a $\mathcal{C}$-diagram of fibrant discrete $G$-spectra. Then 
the assertion of Corollary 
\ref{maincor} that 
$\colim_{N \vartriangleleft_o G} (\holim_\mathcal{C} X)^N$ is a fibrant discrete 
$G$-spectrum is equivalent to claiming that $\mathcal{F}$ is a globally 
fibrant presheaf of spectra (see \cite[pg. 333]{cts}). Also, 
by Lemma \ref{isopresheaves}, to show that $\mathcal{F}$ is globally 
fibrant, it suffices to show that $\holim_\mathcal{C} \mathrm{Hom}_G(-,X)$ is a 
globally fibrant presheaf. This can be done by adapting 
\cite[Proposition 3.3]{Jardinejpaa} and \cite[Lemma 7.3]{cts}, since 
$\mathrm{Hom}_G(-,X)$ is a $\mathcal{C}$-diagram of globally fibrant presheaves 
of spectra. This gives a somewhat different way of obtaining Corollary \ref{maincor}.
\end{Rk}

\section{The explicit construction of a fibrant discrete $G$-spectrum}\label{model}

\par
In this section, we use Corollary \ref{maincor} to construct the fibrant object in 
$\mathrm{Spt}_G$ 
that is our primary object of interest. We begin with several definitions that are 
standard in the theory of discrete $G$-modules and discrete $G$-spectra.

\begin{Def}\label{map}
If $A$ is an abelian group with the discrete topology, let 
$\mathrm{Map}_c(G,A)$ be the abelian group of continuous maps from $G$ to $A$. 
If $Z$ is a spectrum, one can also define the discrete $G$-spectrum 
$\mathrm{Map}_c(G,Z),$ where 
the $l$-simplices $(\mathrm{Map}_c(G,Z)_k)_l$ of the $k$th simplicial set of the spectrum $\mathrm{Map}_c(G,Z)$ are 
given by $\mathrm{Map}_c(G,(Z_k)_l),$ the set of continuous maps from $G$ to 
$(Z_k)_l$. Here, $(Z_k)_l$ is given the discrete 
topology and the $G$-action on $\mathrm{Map}_c(G,Z)$ 
is induced on the level of sets 
by $(g \cdot f)(g') = f(g'g),$ for $g, g' \in G$ and 
$f \in \mathrm{Map}_c(G,(Z_k)_l),$ for each $k, l \geq 0.$ This action 
also makes $\mathrm{Map}_c(G,A)$ a discrete $G$-module.  

\end{Def}

\begin{Def}\label{triple}
Consider the functor 
\[\Gamma_G \colon \mathrm{Spt}_G \rightarrow \mathrm{Spt}_G, \ \ \ 
X \mapsto \Gamma_G(X) = \mathrm{Map}_c(G,X),\] where $\Gamma_G(X)$ has 
the $G$-action given by Definition \ref{map}. As explained 
in \cite[Definition 7.1]{cts}, the functor $\Gamma_G$ 
forms a triple and there is 
a cosimplicial discrete $G$-spectrum $\Gamma^\bullet_G X,$ 
where, for all $n \geq 0$, 
\[(\Gamma^\bullet_G X)^n \cong \mathrm{Map}_c(G^{n+1},X).\] Here, the spectrum 
$\mathrm{Map}_c(G^{n+1},X)$ is defined as in Definition \ref{map}, since 
the cartesian product $G^{n+1}$ is a profinite group, and its discrete $G$-action is given by the $G$-action on 
the constituent sets that is given by 
\[(g \cdot f)(g_1, g_2, g_3, ..., g_{n+1}) = f(g_1g, g_2, g_3, ..., g_{n+1}).\]
\end{Def}

\par
The next definition restates Definition \ref{triple} in the context of discrete 
$G$-modules. 

\begin{Def}
Let $\mathrm{DMod}(G)$ be the category of discrete $G$-modules. 
Then, as in Definition \ref{triple}, there is a functor 
\[\Gamma_G \colon \mathrm{DMod}(G) \rightarrow \mathrm{DMod}(G), \ \ \ 
M \mapsto \Gamma_G(M) = \mathrm{Map}_c(G,M),\] 
and, given a discrete $G$-module $M$, there
is a cosimplicial discrete $G$-module $\Gamma^\bullet_G M.$ 
\end{Def}

\begin{Def}[{\cite[Remark 7.5]{cts}}]
Given a discrete $G$-spectrum $X$, let 
\[\widehat{X} = \colim_{N \vartriangleleft_o G} (X^N)_\mathtt{f}.\] 
Notice that $\widehat{X}$ is a discrete $G$-spectrum, 
since functorial fibrant replacement in $\mathrm{Spt}$ (see the Introduction) 
implies that each 
$(X^N)_\mathtt{f}$ is a $G/N$-spectrum. Also, $\widehat{X}$ 
is fibrant as a spectrum and there is a weak equivalence 
\[\psi \: X \cong \colim_{N \vartriangleleft_o G} X^N \rightarrow 
\colim_{N \vartriangleleft_o G} (X^N)_\mathtt{f} = \widehat{X}\] 
that is $G$-equivariant. 
\end{Def}

\par
We define some useful terminology. If $X^\bullet$ is a cosimplicial object in $\mathrm{Spt}_G$, then $X^\bullet$ is a 
{\em cosimplicial discrete $G$-spectrum}. If $X^\bullet$ is a cosimplicial 
discrete $G$-spectrum such that $X^n$ is a fibrant discrete $G$-spectrum, 
for all $n \geq 0$, then $X^\bullet$ is a {\em cosimplicial fibrant discrete $G$-spectrum}.

\par
The following result defines the explicit globally fibrant object that is of particular 
interest to us. 

\begin{Thm}\label{fibrant}
Let $G$ be a profinite group, and let $H$ be a 
closed subgroup of $G$. If $X$ is a discrete 
$G$-spectrum, then the discrete $H$-spectrum 
\[\colim_{K \vartriangleleft_o H} (\holim_\Delta 
\Gamma^\bullet_G 
\widehat{X})^K\] is fibrant in the model category of discrete $H$-spectra. In particular, 
the discrete $G$-spectrum \[\colim_{N \vartriangleleft_o G} 
(\holim_\Delta \Gamma^\bullet_G \widehat{X})^N\] is fibrant in 
$\mathrm{Spt}_G$.
\end{Thm}
\begin{proof}
By Corollary \ref{maincor}, we only need to show that 
$\Gamma^\bullet_G \widehat{X}$ is a cosimplicial fibrant discrete $H$-spectrum. 
By \cite[Proposition 1.3.4 (c)]{Wilson}, there is a homeomorphism 
$h \: H \times G/H \rightarrow G$ that is $H$-equivariant, where $H$ acts on the 
source by acting only on the factor $H$ and $G/H$ is the 
profinite space $G/H \cong \lim_{N \vartriangleleft_o G} G/NH.$

\par
Given a spectrum $Z$ and a profinite space $W = \lim_\alpha W_\alpha$, where 
each $W_\alpha$ is a finite discrete space, as in Definition \ref{map}, we can 
form the spectrum $\mathrm{Map}_c(W,Z)$, where 
$(\mathrm{Map}_c(W,Z)_k)_l = \mathrm{Map}_c(W,(Z_k)_l)$, and there is an 
isomorphism \[\mathrm{Map}_c(W,Z) \cong 
\colim_\alpha \textstyle\prod_{w \in W_\alpha} Z.\]
Thus, 
\begin{equation}\zig\label{one}
\mathrm{Map}_c(G/H,\widehat{X}) \cong 
\colim_{N \vartriangleleft_o G} \textstyle\prod_{G/NH} \widehat{X}.
\end{equation}
Since filtered colimits commute with finite limits,
\[\textstyle\prod_{G/NH} \widehat{X} \cong 
\textstyle\prod_{G/NH} \displaystyle\colim_{N' \vartriangleleft_o G} (\widehat{X})^{N'} 
\cong \colim_{N' \vartriangleleft_o G} 
(\textstyle\prod_{G/NH} \widehat{X})^{N'},\] and 
it follows that 
$\mathrm{Map}_c(G/H,\widehat{X})$ is a discrete $G$-spectrum, with $G$ acting 
only on $\widehat{X}$. Therefore, by applying the homeomorphism $h$, 
there is an isomorphism
\begin{equation}\zig\label{two}
\mathrm{Map}_c(G, \widehat{X}) \cong 
\mathrm{Map}_c(H, \mathrm{Map}_c(G/H,\widehat{X}))
\end{equation} of discrete $H$-spectra. 

\par
Recall from \cite[Corollary 3.8]{cts} that if $Z$ is a fibrant spectrum, then 
$\mathrm{Map}_c(H,Z)$ is a fibrant discrete $H$-spectrum. Also, since 
$\widehat{X}$ is a fibrant spectrum, the product $\prod_{G/NH} \widehat{X}$ is 
also a fibrant spectrum, so that, by (\ref{one}), $\mathrm{Map}_c(G/H, \widehat{X})$ 
is fibrant in $\mathrm{Spt}$. Then, by applying these observations to (\ref{two}), we 
obtain that $\mathrm{Map}_c(G, \widehat{X})$  is a fibrant discrete $H$-spectrum. 
Hence, $\mathrm{Map}_c(G, \widehat{X})$ is a fibrant spectrum, by 
\cite[Lemma 3.10]{cts}, so that $\mathrm{Map}_c(G, \mathrm{Map}_c(G, \widehat{X}))$ 
is a fibrant discrete $H$-spectrum, by applying the previous argument again. Thus, 
iteration of this argument shows that 
$\Gamma^\bullet_G \widehat{X}$ is a cosimplicial fibrant discrete $H$-spectrum.
\end{proof}

\section{Completing the proof of the main result}

\par
In this section, we finish the proof of the main result. Also, we discuss several 
consequences of having a concrete model for $X_{f,G}$, given a discrete $G$-spectrum 
$X$.

\begin{Def}\label{vcd}
Let $H_c^*(G;M)$ denote 
the continuous cohomology of $G$ with coefficients in the discrete $G$-module $M$. 
Then a profinite group $G$ has {\em finite virtual cohomological dimension} (or 
{\em finite vcd}) if 
there exists an open 
subgroup $H$ of $G$ 
and a non-negative integer $m$, such that $H^s_c(H;M)=0$, for all 
discrete $H$-modules $M$ and all $s \geq m$. 
\end{Def}

\par
Many of the profinite groups that one works with, in practice, 
have finite vcd. For example, if $G$ is a compact $p$-adic analytic group, 
$G$ has finite vcd (see the 
discussion in \cite[pg. 330]{cts}).

\par
Let $X$ be a discrete $G$-spectrum. Then there is a $G$-equivariant 
monomorphism $i \: X \rightarrow \mathrm{Map}_c(G,X)$ that is defined, on 
the level of sets, by $i(x)(g) = g \cdot x$. Then $i$ induces a map $X \rightarrow 
\Gamma^\bullet_G X$ of cosimplicial discrete $G$-spectra, where, here, $X$ is the 
constant diagram. Thus, the composition 
\[X \overset{\psi}{\rightarrow} \widehat{X} \overset{\cong}{\rightarrow} 
\lim_\Delta \widehat{X} 
\rightarrow \holim_\Delta \widehat{X} 
\rightarrow \holim_\Delta \Gamma^\bullet_G \widehat{X}\] of canonical maps 
defines the $G$-equivariant map 
\[\widehat{\psi} \: X \rightarrow \holim_\Delta \Gamma^\bullet_G \widehat{X}\] 
(the canonical map $\lim_\Delta \widehat{X} 
\rightarrow \holim_\Delta \widehat{X}$ is defined in 
\cite[Example 18.3.8 (2)]{Hirschhorn}).

\par
Note that there is a homotopy spectral sequence 
\[E_2^{s,t} = \pi^s(\pi_t(\Gamma^\bullet_G \widehat{X})) \Rightarrow 
\pi_{t-s}(\holim_\Delta 
\Gamma^\bullet_G \widehat{X}),\] where $E_2^{0,t} \cong \pi_t(X)$ and 
$E_2^{s,t} = 0$, when $s > 0$, by \cite[Section 7]{cts}. Thus, the spectral sequence 
collapses, so that the map $\widehat{\psi}$ is a weak equivalence.

\par
Now let $H$ be a closed subgroup of $G$. Then $X$ is a discrete $H$-spectrum, 
so that $X \cong \colim_{K \vartriangleleft_o H} X^K$. Composing this 
isomorphism with the map $\colim_{K \vartriangleleft_o H} (\widehat{\psi})^K$ gives the 
$H$-equivariant map \[\Psi \: X \rightarrow \colim_{K \vartriangleleft_o H} 
(\holim_\Delta 
\Gamma^\bullet_G \widehat{X})^K.\]

\par
Now we show that if $G$ has finite vcd, then $\Psi$ is a weak equivalence. 
As mentioned in the Introduction, the proof (below) closely follows the proof of 
\cite[Theorem 7.4]{cts}, so that our proof will be somewhat 
abbreviated. Also, we should 
mention that the proof of \cite[Theorem 7.4]{cts} follows the arguments given in 
\cite[proof of Proposition 3.3]{Jardinejpaa} and \cite[Proposition 3.20]{Mitchell}.

\begin{Thm}\label{main}
Let $G$ have finite vcd, let $X$ be a discrete $G$-spectrum, and let $H$ be a closed subgroup of $G$. 
Then the map
\[\Psi \: X \rightarrow 
\colim_{K \vartriangleleft_o H} (\holim_\Delta \Gamma^\bullet_G \widehat{X})^K\] is a weak equivalence 
in the category of discrete $H$-spectra, such that the target is 
a fibrant discrete $H$-spectrum.
\end{Thm}
\begin{proof}
Because of the earlier Theorem \ref{fibrant}, we 
only have to prove that $\Psi$ is a weak equivalence of spectra. 
\par
Since $H$ is closed in $G$, $H$ also has finite vcd. Hence, $H$ has a collection 
$\{U\}$ of open normal subgroups such that (a) $\{U\}$ is a cofinal subcollection of 
$\{K\}_{K \vartriangleleft_o H}$ (so, for example, $H \cong \lim_{U} H/{U}$) 
and (b) for all $U$, $H^s_c(U; M) = 0,$ for all
$s \geq m$, where $m$ is some natural number that is independent of $U$, 
and for all discrete $U$-modules $M$. Thus, 
\begin{equation}\zig\label{isomorphism}
\colim_{K \vartriangleleft_o H} 
(\holim_\Delta \Gamma^\bullet_G \widehat{X})^K 
\cong \colim_{U} (\holim_\Delta \Gamma^\bullet_G \widehat{X})^U,
\end{equation}
so that, to show that $\Psi$ is a weak equivalence, it suffices to show that 
the map 
\[\widehat{\Psi} \: X \rightarrow 
\colim_{U} (\holim_\Delta \Gamma^\bullet_G \widehat{X})^U 
\cong \colim_{U} 
\holim_\Delta (\Gamma^\bullet_G \widehat{X})^U,\] 
induced by $\Psi$ and (\ref{isomorphism}), is a weak equivalence.
\par
Notice that each $U$ is a closed subgroup of $G$. Then, 
for each $U$, 
$(\Gamma_G^\bullet \widehat{X})^U$ is a cosimplicial 
fibrant
spectrum, so that there is a conditionally
convergent homotopy spectral
sequence 
\begin{equation}\zig\label{specseq} 
E_2^{s,t}(U) = \pi^s \pi_t ((\Gamma_G^\bullet \widehat{X})^{U})
\Rightarrow \pi_{t-s}(\holim_\Delta 
(\Gamma_G^\bullet \widehat{X})^{U}),
\end{equation} 
with 
\[E_2^{s,t}(U) \cong H^s_c(U; \pi_t(X))\] 
(these assertions are verified in the proof of \cite[Lemma 7.12]{cts}).
\par
Since
$E_2^{s,\ast}(U) = 0$ whenever $s \geq m$, the $E_2$-terms
$E_2^{\ast, \ast}(U)$ are uniformly
bounded on the right. Therefore, by \cite[Proposition 3.3]{Mitchell}, 
taking a colimit over $\{U \}$ of
the spectral sequences in (\ref{specseq}) gives the spectral sequence
\begin{equation}\zig\label{specseq2} 
E_2^{s,t} = \colim_{U} H^s_c(U; \pi_t(X))
\Rightarrow \pi_{t-s}(\colim_{U} \holim_\Delta 
(\Gamma^\bullet_G \widehat{X})^{U}).
\end{equation}  
Notice that \[E_2^{*,t} \cong H^\ast_c(\lim_U U; 
\pi_t(X)) \negthinspace \cong \negthinspace H^\ast_c(\{e\};
\pi_t(X)),\] which is isomorphic to $\pi_t(X)$, concentrated 
in
degree zero.
Thus,
spectral sequence (\ref{specseq2}) collapses, so that, for all $t$, 
$\pi_{t}(\colim_U \holim_\Delta (\Gamma_G^\bullet \widehat{X})^{U}) \cong
\pi_t(X),$ and, hence, $\widehat{\Psi}$ is a weak equivalence.
\end{proof}

\par
Let $X$ be a $\mathcal{C}$-diagram of discrete $G$-spectra, where $\mathcal{C}$ 
is a small category. Then, by Theorem \ref{iso}, there is a canonical map 
\[\phi(X,G) \: \holimG_\mathcal{C} X \cong 
\colim_{N \vartriangleleft_o G} (\holim_\mathcal{C} X)^N  
\rightarrow \holim_\mathcal{C} X\] 
that is $G$-equivariant.

\begin{Cor}
Let $G$ have finite vcd, let $X$ be a discrete $G$-spectrum, and let $H$ be a closed subgroup of $G$. Then the $H$-equivariant map 
\[\phi(\Gamma_G^\bullet \widehat{X}, H) \: 
\holimH_\Delta \Gamma^\bullet_G \widehat{X} \rightarrow 
\holim_\Delta \Gamma^\bullet_G \widehat{X}\] is a weak equivalence in $\mathrm{Spt}$.
\end{Cor} 

\begin{proof} 
Notice that $\widehat{\psi} = \phi(\Gamma_G^\bullet \widehat{X}, H) \circ \Psi$. 
Then the desired conclusion follows from the fact that $\widehat{\psi}$ and $\Psi$ are 
weak equivalences, where the latter fact is from Theorem \ref{main}. 
\end{proof} 

\par
In the Introduction, we pointed out that 
a weak equivalence $r^G_H \: X_{f,G} \rightarrow 
X_{f,H}$ in $\mathrm{Spt}_H$ always exists. The following result uses Theorem \ref{main} to give a concrete 
model for $r^G_H$.

\begin{Cor}\label{mapGH}
Let $G$ have finite vcd, let $X$ be a discrete $G$-spectrum, and let $H$ be a 
closed subgroup of $G$. Then there is a weak equivalence 
\[r^G_H \: 
\colim_{N \vartriangleleft_o G} (\holim_\Delta \Gamma_G^\bullet \widehat{X})^N 
\rightarrow \colim_{K \vartriangleleft_o H} 
(\holim_\Delta \Gamma_G^\bullet \widehat{X})^K\]
in $\mathrm{Spt}_H$, where the source of this map is a fibrant discrete 
$G$-spectrum and the target is a fibrant discrete $H$-spectrum.
\end{Cor}

\begin{proof}
Let $N$ be an open normal subgroup of $G$. Then $N \cap H$ is an open 
normal subgroup of $H$ and, hence, there is a canonical map 
\[(\holim_\Delta \Gamma_G^\bullet \widehat{X})^N \hookrightarrow 
(\holim_\Delta \Gamma_G^\bullet \widehat{X})^{N \cap H} 
\rightarrow \colim_{K \vartriangleleft_o H} 
(\holim_\Delta \Gamma_G^\bullet \widehat{X})^K.\] These maps, as $N$ varies, 
induce the desired map, which is easily seen to be $H$-equivariant. In 
$\mathrm{Spt}_H$, the weak 
equivalence $X \rightarrow \colim_{K \vartriangleleft_o H} 
(\holim_\Delta \Gamma_G^\bullet \widehat{X})^K$ is the composition of 
the weak equivalence $X \rightarrow \colim_{N \vartriangleleft_o G} 
(\holim_\Delta \Gamma_G^\bullet \widehat{X})^N$ and $r^G_H$, so that 
$r^G_H$ is a weak equivalence. 
\end{proof}

\par
The following result is 
a special case of the fact that, if $H$ is open in $G$, then a fibrant discrete $G$-spectrum is also fibrant 
as a discrete $H$-spectrum (see \cite[Lemma 3.1]{iterated} and 
\cite[Remark 6.26]{Jardine}).

\begin{Cor}\label{fibrancy}
Let $G$ have finite vcd and let $X$ be a discrete $G$-spectrum. 
If $H$ is an open subgroup of $G$, then 
$\colim_{N \vartriangleleft_o G} 
(\holim_\Delta (\Gamma^\bullet_G 
\widehat{X}))^N$, a fibrant discrete $G$-spectrum, 
is also a fibrant discrete $H$-spectrum.
\end{Cor}

\begin{proof} 
By Theorem \ref{main}, the spectrum 
$\colim_{K \vartriangleleft_o H} 
(\holim_\Delta (\Gamma^\bullet_G 
\widehat{X}))^K$ is a fibrant discrete $H$-spectrum. Thus, to verify the corollary, 
it suffices to show that this fibrant discrete $H$-spectrum is isomorphic to 
$\colim_{N \vartriangleleft_o G} (\holim_\Delta (\Gamma^\bullet_G \widehat{X}))^N$ 
in $\mathrm{Spt}_H$. 

\par
Note 
that if $U$ is an open subgroup 
of $H$, then $U$ is also an open subgroup of $G$, so that 
\[\{H \cap V \, | \, V <_o G\} = \{U \, | \, U <_o H\}.\] Also, 
$\{ H \cap V \, | \, V <_o G \}$ and $\{N \, | \, N \vartriangleleft_o G \}$ are cofinal subcollections of the set $\{V \, | \, V<_oG \}$, so that  
\begin{align*}
\colim_{N \vartriangleleft_o G} (\holim_\Delta (\Gamma^\bullet_G 
\widehat{X}))^N & \cong \colim_{V <_o G} 
(\holim_\Delta (\Gamma^\bullet_G 
\widehat{X}))^V \\ & \cong \colim_{V <_o G} 
(\holim_\Delta (\Gamma^\bullet_G 
\widehat{X}))^{H \cap V} \\ & = \colim_{U <_o H} 
(\holim_\Delta (\Gamma^\bullet_G 
\widehat{X}))^U \\ & \cong
\colim_{K \vartriangleleft_o H} 
(\holim_\Delta (\Gamma^\bullet_G 
\widehat{X}))^K.
\end{align*}  
\end{proof}

\section{Understanding $X^{hH}$ when $H$ is closed, and not open, in $G$}\label{hH}
In this section, we use the explicit fibrant model of Theorem \ref{main} 
to improve (somewhat) our 
understanding of $X^{hH}$, for $X \in \mathrm{Spt}_G$, when $G$ has finite 
vcd and $H$ is a closed non-open subgroup of $G$.  

\par
In the Introduction, we pointed out that 
if $L$ is any profinite group, $Z \in \mathrm{Spt}_L$, $Z \rightarrow Z_{f,L}$ is a trivial cofibration, and 
$Z \rightarrow Z^{f,L}$ is a weak equivalence, with $Z_{f,L}$ and $Z^{f,L}$ fibrant 
- all in $\mathrm{Spt}_L$, then $Z^{hL}=(Z_{f,L})^L \rightarrow (Z^{f,L})^L$ is a weak 
equivalence, so that we can identify $Z^{hL}$ and $(Z^{f,L})^L$. For the upcoming 
discussion, we make this identification explicit. 

\begin{Def}\label{def}
If $L$ is a profinite group and $Z$ is a discrete $L$-spectrum, we 
define $Z^{hL} = (Z_{f,L})^L,$ where $Z \rightarrow Z_{f,L}$ is a weak equivalence 
and $Z_{f,L}$ is fibrant, all in $\mathrm{Spt}_L$.
\end{Def}

\begin{Thm}\label{H}
If $G$ has finite vcd, $H$ is a closed subgroup of $G$, and $X$ is a discrete 
$G$-spectrum, then 
\[X^{hH} \cong  (\holim_\Delta (\Gamma^\bullet_G 
\widehat{X}))^H.\] 
\end{Thm}

\begin{proof}
By Definition \ref{def} and Theorem \ref{main}, 
\[X^{hH} = (\colim_{K \vartriangleleft_o H} 
(\holim_\Delta (\Gamma^\bullet_G 
\widehat{X}))^K)^H.\] As in the proof of Theorem \ref{iso}, since $H$ is 
an open normal subgroup of itself, this expression simplifies to give the 
desired result.
\end{proof}

\begin{Rk}
Suppose that the hypotheses of Theorem \ref{H} are satisfied. 
In \cite[Remark 7.13]{cts}, by using a different argument, we noted 
that there are certain weak equivalences that permit 
$(\holim_\Delta (\Gamma^\bullet_G 
\widehat{X}))^H$ to be taken as a definition of $X^{hH}$. (To be precise, 
\cite[Remark 7.13]{cts} uses ``$X_{f,G}$" instead of $\widehat{X}$ in 
the expression $(\holim_\Delta (\Gamma^\bullet_G 
\widehat{X}))^H$, but it is not 
hard to see that the remark is still valid with $\widehat{X}$ in place of ``$X_{f,G}$.") 
However, Theorem 
\ref{H} puts this definition on firmer ground by showing that it comes from taking 
the $H$-fixed points of a fibrant replacement for $X$.
\end{Rk}

\par
For this paragraph and the next, let $X$ be a discrete $G$-spectrum and let 
$H$ be a closed proper normal subgroup of $G$, such that $H \neq \{e\}$. 
Then $G/H$ is a profinite group and 
the $H$-fixed point spectrum $X^H$ is a $G/H$-spectrum. However, as noted in 
\cite[Sections 1 and 3]{iterated}, the corresponding 
situation is more complicated with $H$-homotopy 
fixed points. It is natural to expect 
$X^{hH}$ to be a $G/H$-spectrum, but, given an arbitrary fibrant replacement 
$X_{f,H}$ (as guaranteed by the model category axioms), in general, there is no 
reason to assume that $X_{f,H}$ carries a $G$-action, so that $(X_{f,H})^H$ 
need not be a $G/H$-spectrum. For example, if $H$ has finite vcd, 
$X_{f,H}$ can be taken to be $\colim_{K \vartriangleleft_o H} 
(\holim_\Delta (\Gamma^\bullet_H 
\widehat{X}))^K$, a discrete $H$-spectrum that is a $G$-spectrum only when 
additional hypotheses are satisfied.

\par
If $H$ is open in $G$, then, as recalled at the end of the preceding section, 
$X_{f,G}$ is a fibrant discrete $H$-spectrum, so that $X^{hH} = (X_{f,G})^H$ is 
easily seen to be a $G/H$-spectrum. However, if $H$ is not open in $G$, then 
the situation is much more complicated: it is natural to wonder if the map 
\[(r^G_H)^H \: (X_{f,G})^H \rightarrow (X_{f,H})^H = X^{hH}\] 
is a weak equivalence, but the 
evidence indicates that it does not have to be (even when $G$ has finite vcd), since, for example, $X_{f,G}$ is not 
known to be fibrant in $\mathrm{Spt}_H$. (There is no example known 
of $(r^G_H)^H$ failing to be a weak equivalence, but there are several 
arguments indicating that there should be such examples - see 
\cite[Sections 1, 3, and 4]{iterated} and \cite[Sections 3.5 and 3.6]{joint} for 
more discussion on this issue.) 
Thus, the theory of homotopy 
fixed points is faced with the undesirable fact that, in general, when $H$ is not 
open in $G$, it is not known 
how to show that $X^{hH}$ is a $G/H$-spectrum. However, 
Theorem \ref{H} immediately implies the following result, for the case when $G$ 
has finite vcd.  

\begin{Cor}\label{G/H}
If $G$ has finite vcd, with $H$ a closed normal subgroup of $G$, and $X$ is a discrete 
$G$-spectrum, then $X^{hH}$, when taken to be 
\[X^{hH} \cong  (\holim_\Delta (\Gamma^\bullet_G \widehat{X}))^H,\] 
is a $G/H$-spectrum.
\end{Cor}

\par
In the context of Corollary \ref{G/H}, it is natural to inquire about the $G/H$-homotopy 
fixed points of $X^{hH}$. But there are nontrivial issues involved here; we refer 
the reader to \cite{iterated} for more on this topic.


\begin{thebibliography}{99}

\bibitem{joint}
Mark Behrens and Daniel~G. Davis, 
{\em The homotopy fixed point spectra of profinite {G}alois extensions},
\newblock in preparation, available at each author's homepage.

\bibitem{Bousfield/Kan}
A.~K. Bousfield and D.~M. Kan, {\em Homotopy Limits, Completions and
  Localizations}, Springer-Verlag, Berlin, 1972, Lecture Notes in Mathematics,
  Vol. 304.
  
\bibitem{cts}
Daniel~G. Davis, {\em Homotopy fixed points for {$L_{K(n)}(E_n \wedge X)$}
  using the continuous action}, J. Pure Appl. Algebra {\bf 206} (2006), no.~3,
  322--354.
  
\bibitem{iterated}
Daniel~G. Davis, {\em Iterated homotopy fixed points for the {L}ubin-{T}ate spectrum},
\newblock submitted for publication, available online as
  arXiv:math.AT/0610907.
  
\bibitem{hGal}
Paul~G. Goerss.
\newblock Homotopy fixed points for {G}alois groups.
\newblock In {\em The \v Cech centennial (Boston, MA, 1993)}, pages 187--224.
  Amer. Math. Soc., Providence, RI, 1995.  

\bibitem{GJ}
Paul~G. Goerss and John~F. Jardine, {\em Simplicial Homotopy Theory},
  Birkh\"auser Verlag, Basel, 1999.
  
\bibitem{GJlocal}
P.~G. Goerss and J.~F. Jardine.
\newblock Localization theories for simplicial presheaves.
\newblock {\em Canad. J. Math.}, 50(5):1048--1089, 1998.

\bibitem{Hirschhorn}
Philip~S. Hirschhorn.
\newblock {\em Model categories and their localizations}, volume~99 of {\em
  Mathematical Surveys and Monographs}.
\newblock American Mathematical Society, Providence, RI, 2003.

\bibitem{Jardine}
J.~F. Jardine.
\newblock {\em Generalized \'etale cohomology theories}.
\newblock Birkh\"auser Verlag, Basel, 1997.

\bibitem{Jardinesummary}
J.~F. Jardine.
\newblock Generalised sheaf cohomology theories.
\newblock In {\em Axiomatic, enriched and motivic homotopy theory}, volume 131
  of {\em NATO Sci. Ser. II Math. Phys. Chem.}, pages 29--68. Kluwer Acad.
  Publ., Dordrecht, 2004.   

\bibitem{Jardinejpaa}
J.~F. Jardine.
\newblock Simplicial presheaves.
\newblock {\em J. Pure Appl. Algebra}, 47(1):35--87, 1987.

\bibitem{Mitchell}
Stephen~A. Mitchell, {\em Hypercohomology spectra and {T}homason's descent
  theorem}, Algebraic $K$-Theory (Toronto, ON, 1996) (V. Snaith, ed.), Fields Inst. Commun.,
  vol.~16, Amer. Math. Soc., Providence, RI, 1997, pp.~221--277.
  
\bibitem{Thomason}
R.~W. Thomason.
\newblock Algebraic ${K}$-theory and \'etale cohomology.
\newblock {\em Ann. Sci. \'Ecole Norm. Sup. (4)}, 18(3):437--552, 1985.  
  
\bibitem{Wilson}
John~S. Wilson.
\newblock {\em Profinite groups}.
\newblock The Clarendon Press Oxford University Press, New York, 1998.
  
\end{thebibliography}
\end{document}